\documentclass[11pt]{amsart}
\usepackage{graphicx}
\usepackage{latexsym}
\usepackage{amsfonts,amsmath,amssymb}
\newtheorem{theorem}{Theorem}[section]

\newtheorem{lemma}[theorem]{Lemma}

\newtheorem{corollary}[theorem]{Corollary}

\newtheorem{conjecture}[theorem]{Conjecture}

\def\pr{\text{P}}
\def\ex{\text{E}}
\def\eps{\varepsilon}

\def\cal{\mathcal}

\date{\today}

\begin{document}

\title[Random search tree]{On a random search tree: asymptotic enumeration of vertices by distance from leaves}

\author{Mikl\'os B\'ona}
\address{Department of Mathematics, University of Florida, $358$ Little Hall, PO Box $118105$,
Gainesville, FL, $32611-8105$ (USA)}
\email{bona@ufl.edu}
\author{Boris Pittel}
\address{Department of Mathematics, The Ohio State University, $231$ West $18$-th Avenue, Columbus, Ohio $43210-1175$ (USA)}
\email{bgp@math.ohio-state.edu}



\begin{abstract} A random binary search tree grown from the uniformly random permutation of
$[n]$ is studied. We analyze the exact and asymptotic counts of vertices by  rank, the distance
from the set of leaves. The asymptotic fraction $c_k$ of vertices of a fixed rank $k\ge 0$ is shown to decay exponentially with $k$. Notoriously hard to compute, the exact fractions $c_k$ had been
determined for $k\le 3$ only.  We  computed $c_4$ and
$c_5$ as well; both are ratios of enormous integers, denominator of $c_5$ being $274$
digits long. Prompted by the data, we proved that, in sharp contrast, the largest prime divisor
of $c_k$'s denominator is $2^{k+1}+1$ at most.  We conjecture that, in fact, 
 the prime divisors of every denominator for $k>1$ form a single interval,  from $2$ to the
 largest prime not exceeding $2^{k+1}+1$.
\end{abstract}
\maketitle
\section{Introduction}
\subsection{Background and Definitions}
Various parameters of many models of random rooted trees are fairly well understood {\em if they relate to a near-root part of the tree or to global tree structure }. The first group includes, for instance, the numbers of vertices at given distances from the root, the immediate progeny sizes
for vertices near the top, and so on. See Flajolet and Sedgewick \cite{flajolet} for a comprehensive treatment of these results. The tree height and width are parameters of global nature, see Kolchin \cite{kol},
Devroye \cite{dev}, Mahmoud and Pittel \cite{mahmoudpit}, Pittel \cite{averagerootheight},
Kesten and Pittel \cite{kespit}, Pittel \cite{growtrees}, for instance. {\em Profiles} of random trees have been studied in
\cite{drmota} and \cite{pitman}. In recent years there has
been a growing interest in analysis of the random tree fringe, i. e. the  tree part close to the leaves, 
see Aldous \cite{aldous}, Mahmoud and Ward \cite{mahmoudward1}, \cite{mahmoudward2},
B\'ona \cite{protected}, and Devroye and Janson \cite{janson}. Diversity of models and techniques
notwithstanding, a salient feature of these studies is usage of inherently recursive nature of
the random trees in question. Deletion of the tree root produces a forest of rooted subtrees that are conditionally independent, each being distributed as the random tree for the properly chosen tree size.

Not surprisingly, the technical details of fringe analysis become quite complex as soon
as the focus shifts to layers of vertices further away from the leaves. So while there are
explicit results on the (limiting) fraction of vertices at a fixed, small, distance from the
leaves, an asymptotic behavior of this fraction, as a function of the distance, remained an open problem. In the present paper we will solve this problem for the random 
{\em decreasing binary trees}, known also as {\em binary search trees}. We hope to study
other random trees in a subsequent paper. 

A decreasing binary tree on vertex set $[n]=\{1,2,\cdots ,n\}$ is a binary plane tree in which every vertex has a smaller label
than its parent. Note that this means that the root must have label $n$. Also note that every vertex has at most two 
children, and that every child $v$ is either an left child or a right child of its parent, even if $v$ is the only child of its parent. 

Decreasing binary trees on vertex set $[n]$ are in bijection with permutations of $[n]$. In order to see this, 
let $p=p_1p_2\cdots p_n$ be a permutation. The  decreasing binary tree
 of $p$, which
we denote by $T(p)$, is defined as follows. The root of $T(p)$ is a vertex
 labeled $n$, the largest entry of $p$. 
If $a$ is the largest entry of $p$ on the left of $n$, and $b$ is the largest
entry of $p$ on the right of $n$, then the root will have two children,
the left one will be labeled $a$, and the right one labeled $b$. If $n$ is
the first (resp. last) entry of $p$, then the root will have only one child,
and that is a left  (resp. right) child, and it will necessarily be
labeled $n-1$ as $n-1$ must be the largest of all remaining elements.
Define the rest of $T(p)$ recursively,  by taking $T(p')$ and $T(p'')$, where
$p'$ and $p''$ are the substrings of $p$ on the two sides of $n$, and affixing
them to $a$ and $b$.

\subsection{Recent results} \label{recent}
{\em For the rest of this paper, whenever we say {\em tree}, we will mean a decreasing binary tree. }

If $v$ is a vertex of a  tree $T$, then let the {\em rank} of $v$ be the number of edges in the
shortest  path from $v$ to a leaf of $T$ that is a descendant of $v$. So leaves are of rank 0, neighbors of leaves are of rank 1, and so on. Motivated by a series of recent papers \cite{du}, \cite{mahmoudward1}  concerning the neighbors of leaves, Mikl\'os B\'ona \cite{protected},
proved that for any $k\geq 0$, the probability that a randomly selected vertex of a randomly selected tree is 
of rank $k$ converges to a rational number $c_k$ as $n$ goes to infinity. He also computed 
 that $c_0=1/3$, $c_1=3/10$, $c_2=1721/8100$, and 
$c_3\approx 0.105$. It is worth mentioning that a few months later, Svante Janson and Luc Devroye computed the
same four values of $c_k$ with a completely different method. (The numbers $c_k$ are completely determined {\em
theoretically}, but progressively more difficult to compute as $k$ increases.)
These data show that roughly 95.5 percent of all vertices are of rank at most three, and raises the
very intriguing questions whether $\{c_k\}$ is a probability distribution, and if yes, whether
it is the limiting distribution of the rank of the uniformly random vertex of the tree. We were also
keen to find a way for precise evaluation of the next constants, $c_4$ and $c_5$ at least.

\subsection{Main results}\label{main} In this paper, we are able to answer these questions.
Here are our main results.
\begin{theorem}\label{A} {\bf (i)\/} The equality $\sum_{k\ge 0} c_k=1$ holds, and so $\{c_k\}$ is the probability
distribution of a random variable $R$. {\bf (ii)\/} Let $R_n$ be the rank of the uniformly random vertex of the tree. Then for every $0<\rho<3/2$, $\lim_{n\to \infty}\ex\bigl[\rho^{R_n}\bigr]=
\ex\bigl[\rho^R\bigr]<\infty$. Consequently $R_n\to R$ in distribution, and with all its 
moments, and $c_k=O(q^k)$ for every $0<q<2/3$. {\bf (iii)\/}
Let $R_n^{(1)}, \dots R_n^{(t)}$ be the ranks of the uniformly random $t$-tuple  of vertices of
the tree. Then $(R_n^{(1)}, \dots R_n^{(t)})$ converges in distribution to $(R^{(1)},\dots, R^{(t)})$,
with the components $R^{(j)}$ being independent copies of $R$.
\end{theorem}

The part {\bf (ii)\/} is consistent, broadly, with the conjecture in \cite{protected} stating that the sequence $\{c_k\}$ is log-concave. Focusing exclusively on this sequence we show that the decay of $c_k$ is exactly exponentially fast.

To state the result concisely,  introduce the function $g(\alpha)=
\alpha+\alpha\log(2/\alpha)-1$. The equation $g(\alpha)=0$ has two positive roots. Let $\alpha_0$
denote the smaller root; $\alpha_0\approx 0.373$.
\begin{theorem}\label{B} There exists $\gamma>0$ such that for all $k\ge 1$,
\[
\gamma e^{-k/\alpha_0}\le 1-\sum_{j=0}^{k-1}c_j \le \frac{6k+7}{3}\left(\frac{1}{3}\right)^k.
\]
\end{theorem}
\noindent {\bf Note.\/} If  $\lim k^{-1}\log(1/ c_k)$ exists, and we conjecture it does, then this limit is in $[\log 3,1/\alpha_0]$.

We also found a way to simplify computation of the numbers $c_k$ which enabled us to obtain the precise values of $c_4$ and $c_5$, thus going beyond $c_0,\dots,c_3$ determined in \cite{protected} and \cite{janson}. Our numerical results
show that, with high probability, about $99.875$ percent of all vertices are of rank five or less. When written in simplest form,
the numerators and denominators of the rational numbers $c_k$ grow very fast. For instance, the denominator of 
$c_5$ ($denom(c_5)$) has $274$ digits. Despite its enormity, the largest prime divisor of
$denom(c_5)$ is $61$. We conjectured and proved that this remarkable pattern holds for all $k$:
the largest prime divisor of $denom(c_k)$ is at most $2^{k+1}+1$. 
So the $274$-digit denominator of $c_5$ has no prime divisor larger than $65$, i. e.
larger than $61$, which is indeed its prime divisor!  
On the basis of our data, we conjecture that, for $k\ge 2$,  the set of prime divisors of
 $denom(c_k)$ is an {\it uninterrupted\/} interval of primes from $2$ to the largest prime
 divisor, thus (by the prime number theorem) having length $\approx 2^{k+1}/k\log 2$ for large $k$.
That same data makes us believe that the numerator and the denominator of $c_k$ are
comparable in order of magnitude, but the numerator has very few prime factors, with the smallest  one rapidly growing as $k$ increases.

\section{Convergence of the random rank $R_n$}
We start by introducing $E_{n,k}$, the expected number of vertices of
rank $k$. Our focus is on existence and the values
of the limits \[c_k = \lim_{n\rightarrow \infty} \frac{E_{n,k}}{n},\quad k\geq 0.\]
Equivalently, $c_k$ is the limiting probability that $R_n$, the rank of the uniformly random vertex of the (uniformly) random tree is $k$.

The data on $c_k$ that we mentioned in Section \ref{recent}  makes plausible a conjecture that $\{c_k\}$ is actually a probability
distribution, so that there exists a random variable $R$ such that $\pr(R=k)=c_k$  and
$R_n\Rightarrow R$ in distribution. Our first theorem confirms this conjecture with room to spare, 
demonstrating  
that the moment generating function of $R_n$ converges to that of $R$ for any argument
below $3/2$.
\begin{theorem}\label{Rn->R} For every $\rho<3/2$, $\limsup \ex\,[\rho^{R_n}]<\infty$. 
Consequently $\{c_k\}$ is a probability distribution of a random variable $R$ and
$\lim\ex\,[\rho^{R_n}]=\ex\,[\rho^R]$. 
\end{theorem}
\begin{proof}
Let $p_{n,k}$ be the probability that the root is
of rank $k$. Then, for $n>1$, 
\begin{equation} \label{recurrence} 
E_{n,k} = p_{n,k} + \frac{1}{n} \sum_{j=0}^{n-1} \left( E_{j,k} + E_{n-1-j,k} \right), 
\end{equation}
Indeed, the above formula just adds the expected value of indicator of the event ``root is
of rank  $k$'' to the expected total count 
of the non-root vertices of rank $k$, the latter being first computed for trees in which the left subtree of the root is of size $j$. The existence of $c_k:=\lim E_{n,k}/n$, rational or not, will follow immediately from the next lemma. 
\begin{lemma}\label{limexists} Let $\{x_n\}$, $y_n$  be such that 
$y_n=O(n^{1-\eps})$, $(\eps>0)$, and
\begin{equation*}
x_n=y_n+\frac{1}{n}\sum_{j=0}^{n-1}(x_j +x_{n-1-j}),\quad n>1.
\end{equation*}
Then there exists a finite $\lim_{n\to\infty}x_n/n$.
\end{lemma}
\begin{proof} First of all, \eqref{recurrence} is equivalent to
\begin{equation*}
x_n=y_n+\frac{2}{n}\sum_{j=0}^{n-1}x_j,\quad n>1.
\end{equation*}
Standard manipulation shows then that 
\begin{equation}\label{stand}
nx_n-(n+1)x_{n-1}=ny_n-(n-1)y_{n-1},\quad n>1,
\end{equation}
or 
\begin{align*}
\frac{x_n}{n+1}-\frac{x_{n-1}}{n}&=\frac{y_n}{n+1}-\frac{y_{n-1}}{n}\frac{n-1}{n+1}\\
&=\frac{y_n}{n+1}-\frac{y_{n-1}}{n}+O\bigl(n^{-1-\eps}\bigr).
\end{align*}
Telescoping, we obtain: for $1<m<n$,
$$
\frac{x_n}{n+1}-\frac{x_m}{m+1}=\frac{y_n}{n+1}-\frac{y_m}{m+1}+O(m^{-\eps})=O(m^{-\eps}).
$$
Thus $\{x_n/(n+1)\}$ is a fundamental Cauchy sequence, whence there exists a finite $\lim_{n\to\infty} x_n/(n+1)$, and so does $\lim_{n\to\infty}x_n/n$.
\end{proof}
Since $p_{n,k}=O(1)$, the conditions of Lemma \ref{limexists}
obviously hold for $x_n=E_{n,k}$ and $y_n=p_{n,k}$ with $\eps\in (0,1]$. Consequently, for each $k\ge 0$,
there exists a finite limit $c_k:=\lim E_{n,k}/n$. Further,
\begin{equation}\label{sumck}
\sum_k \frac{E_{n,k}}{n}=1\Longrightarrow \sum_k c_k\le 1.
\end{equation}

Next, given $\rho>1$, introduce
\[
{\cal H}_n(\rho)=\sum_{k\le n-1} \rho^k E_{n,k}, 
\]
the expected value of $\sum_{v\in [n]} \rho^{R(v)}$, $R(v)$ denoting
the rank of a generic vertex $v$. Then, analogously to \eqref{recurrence},
\begin{equation}\label{calHtnrecur}
\cal H_n(\rho) =h_n(\rho) + \frac{1}{n} \sum_{j=0}^{n-1} \left( \cal H_j(\rho) + \cal H_{n-1-j}(\rho) \right), \quad n>1,
\end{equation}
where $h_n(\rho)=\ex\bigl[\rho^{R(root)}\bigr]$. How large are $h_n(\rho)$ and $\cal H_n(\rho)$?

Let $X_{n,j}$ denote the random number of leaves at (edge) distance $j$ from the root; $L_n=\sum_jX_{n,j}$ is the total number of leaves. Then
\begin{equation}\label{hn(rho)<}
\rho^{R(root)}\le \frac{\sum_j \rho^jX_{n,j}}{L_n}\Longrightarrow h_n(\rho)\le\ex\left[\frac{\sum_j \rho^jX_{n,j}}{L_n}\right].
\end{equation}
We will  show that  $L_n$ is of order $n$ so it is  likely that $h_n(\rho)$ is at most
of order $n^{-1}\sum_j \rho^j \ex[X_{n,j}]$. So let us bound $\sum_j \rho^j \ex[X_{n,j}]$. To do so, attach to  the random tree ``external'' vertices so that
every vertex of the tree itself  has exactly two descendants; thus every leaf $\ell$ gets two
external descendants, and every non-leaf vertex of the tree with one (left/right) descendant gets 
an additional  external (right/left) descendant. Let $\cal X_{n,j}$ denote the total number of
external nodes at distance $j$ from the root. It was shown in \cite{mahmoudpit} that 
\begin{equation}\label{calLj(x)}
\cal L_j(x): =\sum_{n\ge 1}\ex[\cal X_{n,j}]x^n=\frac{2^j}{j!}\left(\log\frac{1}{1-x}\right)^j,\quad j>0.
\end{equation} 
Introduce $L_j(x)=\sum_{n\ge 0}x^n\ex[X_{n,j}]$; so $L_0(x)= x$. Arguing as in \cite{mahmoudpit}, it can be shown that, for $j\ge 2$,
\[
\frac{d L_j(x)}{dx}=\frac{2}{1-x}\,L_{j-1}(x).
\]
Notice that $[x^n] L_0(x)\leq [x^n] \log\tfrac{1}{1-x}$ for every $n\ge 0$.  By induction on $j$, it follows that, for $j>0$,
\begin{equation}\label{Lj(x)}
 \ex[X_{n,j}]=[x^n] L_j(x)\leq \frac{2^j}{j!}[x^n]\left(\log\frac{1}{1-x}\right)^j=\ex[\cal X_{n,j}].
\end{equation} 

Therefore, for every $r>0$,
\begin{align*}
\sum_j r^j \ex[X_{n,j}]&=[x^n]\sum_{j\ge 0}r^jL_j(x) \le [x^n]\sum_{j\ge 0}r^j\cal L_j(x)\\
&=[x^n]\sum_{j\ge 0}\frac{(2r)^j}{j!}\left(\log\frac{1}{1-x}\right)^j=[x^n]\exp\left[2r \log\frac{1}{1-x}\right]\\
&=[x^n] (1-x)^{-2r}=\binom{n+2r-1}{n}=\frac{\Gamma(n+2r)}{\Gamma(n)\Gamma(2r)}\\
&=O\bigl(n^{2r-1}\bigr),
\end{align*}
the last equality following from the Stirling formula for the Gamma function. Thus, for $r>0$,
\begin{equation}\label{Hrhon<}
\sum_j r^j \ex[X_{n,j}]=O\bigl(n^{2r-1}\bigr).
\end{equation}
Consequently for the numerator in the bound \eqref{hn(rho)<} of  $h_n(\rho)$
we have
\[
\ex\left[\sum_j\rho^jX_{n,j}\right]=O\bigl(n^{2\rho-1}\bigr).
\]
It remains to show that  the denominator $L_n$ in \eqref{hn(rho)<} is quite likely to be
of order $n$, so that $h_n(\rho)=O\bigl(n^{2\rho-1}/n\bigr)=O\bigl(n^{2\rho-2})$. To be more specific,
since $\ex[L_n]=(n+1)/3$, \cite{protected}, we should expect that
$\pr(L_n< an)$ is very small if $a<1/3$.

\begin{lemma}\label{Xn} If $x\in (0,1]$ and $y\in (0, y(x))$, 
\begin{equation}\label{y(x):=}
y(x):=(2\sqrt{1-x})^{-1}\log\tfrac{1+\sqrt{1-x}}{1-\sqrt{1-x}},
\end{equation}
then, setting $L_0=0$,
\begin{equation}\label{bivariate}
\sum_{n\ge 0}y^n \ex\bigl[x^{L_n}\bigr]=\sqrt{1-x}\,\frac{1+e^{2\sqrt{1-x}y}
\tfrac{1-\sqrt{1-x}}{1+\sqrt{1-x}}}{1-e^{2\sqrt{1-x}y}
\tfrac{1-\sqrt{1-x}}{1+\sqrt{1-x}}}.
\end{equation}
\end{lemma}
\begin{proof} Since for $n>1$
\[
\ex\bigl[x^{L_n}\bigr] = \frac{1}{n}\sum_{k=0}^{n-1}\ex\bigl[x^{L_k}\bigr]\ex\bigl[x^{L_{n-1-k}}\bigr],
\]
we obtain
\begin{equation}\label{diff,F1}
\begin{aligned}
\frac{\partial}{\partial y}\sum_{n\ge 0}y^n \ex\bigl[x^{L_n}\bigr]&=x+\sum_{n\ge 2}y^{n-1}
\sum_{k=0}^{n-1}\ex\bigl[x^{L_k}\bigr] \ex\bigl[x^{L_{n-1-k}}\bigr]\\
&=\left(\sum_{n\ge 0}y^n \ex\bigl[x^{L_n}\bigr]\right)^2 - (1-x).
\end{aligned}
\end{equation}
Integrating and using $\left.\sum_{n\ge 0}y^n \ex\bigl[x^{L_n}\bigr]\right|_{y=0}=1$, we obtain 
\eqref{bivariate},
provided that the denominator in \eqref{bivariate} is positive, a condition equivalent to $y<y(x)$.
\end{proof}
\begin{corollary}\label{PXn<an} Let $a<1/3$. For $\delta\in (0,1)$,
\[
\pr(L_n <an)\le \exp\bigl(-(1/3-a )n^{1-\delta}/2\bigr).
\]
\end{corollary}
\begin{proof} We start with a Chernoff-type bound
\begin{equation}\label{Chernoff}
\pr(L_n<an)\le x^{-an} y^{-n} \sum_{\nu\ge 0}y^{\nu} \ex\bigl[x^{L_{\nu}}\bigr],\quad \forall\, x<1,\,
y<y(x).
\end{equation}
Choose $x=\exp\bigl(-n^{-\delta}\bigr)$; then 
\[
y(x)=1+\frac{1}{3n^{\delta}}+O(n^{-2\delta}),
\]
so we may choose $y=\exp\bigl(b n^{-\delta}\bigr)$, $b=(a+1/3)/2$. Using \eqref{Chernoff}
and \eqref{bivariate}, it follows that
\[
\pr(L_n<an)=O\bigl[\exp(an^{1-\delta}-bn^{1-\delta})\bigr]=O\bigl[\exp(-(1/3-a)n^{1-\delta}/2\bigr].
\]
\end{proof}

Armed with the corollary,  we return to \eqref{hn(rho)<}. By Cauchy-Schwartz inequality,
\[
\sum_j\rho^jX_{n,j}=\sum_{\ell} \rho^{|\cal P(\ell)|}\le X_n^{1/2} 
\left(\sum_{\ell}\rho^{2|\cal P(\ell)|}\right)^{1/2}\le n^{1/2}\left(\sum_j \rho^{2j}X_{n,j}\right)^{1/2}.
\]
Therefore, applying Cauchy-Schwartz inequality again and using \eqref{Hrhon<},
\begin{align*}
\ex\left[\bold 1_{\{X_n\le an\}}\sum_j\rho^jX_{n,j}\right]&\le n^{1/2}\left(\ex\bigl[\bold 1_{X_n\le an}\bigr]
\right)^{1/2}\left(\ex\left[\sum_j\rho^{2j}X_{n,j}\right]\right)^{1/2}\\
&=n^{1/2} \pr^{1/2}(X_n\le an)\left(\sum_j\rho^{2j}\ex[X_{n,j}]\right)^{1/2}\\
&=O\bigl[n^{1/2} n^{(2\rho^2-1)/2}\pr^{1/2}(X_n\le an)\bigr]\\
&=O\bigl[n^{\rho^2}\pr^{1/2}(X_n\le an)\bigr]\
\end{align*}
Using the bound \eqref{Hrhon<} with $\rho^2$ instead of $\rho$ and Corollary \ref{PXn<an},
we obtain then
\[
\ex\left[\bold 1_{\{X_n\le an\}}\sum_j\rho^j X_{n,j}\right]= O\bigl(n^{\rho^2} \exp(-(1/3-a)n^{1-\delta}/4)
\bigr)=o(1).
\]
Therefore, by \eqref{hn(rho)<} and \eqref{Hrhon<},
\begin{equation}\label{hn(rho)<expl}
\begin{aligned}
h_n(\rho) &\le \ex\left[\bold 1_{\{X_n<an\}}\sum_j\rho^jX_{n,j}\right]+ \frac{1}{an}\sum_j\rho^j \ex[X_{n,j}]\\
&=o(1)+O\bigl(n^{2\rho-2}\bigr).
\end{aligned}
\end{equation}
\begin{lemma}\label{limH(t)n/nexists} For every fixed $\rho<3/2$, there exists a finite
$\lim_{n\to\infty}n^{-1}\cal H_n(\rho)$. Consequently $\sum_{k\ge 0} c_k=1$, 
$\sum_{k\ge 0} \rho^k c_k <\infty$, and so $c_k=o(\rho^{-k})$.
\end{lemma}
\begin{proof}  By \eqref{Hrhon<} and \eqref{hn(rho)<expl}, $x_n:=\cal H_n(\rho)$ and $y_n:=
h_n(\rho)$ satisfy the condition of Lemma \ref{limexists} with $\eps\in (0,3-2\rho)$. Hence, there exists a finite 
\[
\lim_{n\to\infty}n^{-1}\cal H_n(\rho)=\lim_{n\to\infty}n^{-1}\ex\left[\sum_{v\in [n]}\rho^{R(v)}\right]=
\lim_{n\to\infty}n^{-1}\!\!\!\sum_{k\le n-1}\!\!\!\rho^kE_{n,k}.
\]
Since $n^{-1}\sum_{k\le n-1}E_{n,k}=1$,  and there exists $c_k=\lim_{n\to\infty}n^{-1}E_{n,k}$,
($k\ge 0$), we conclude that $\sum_k c_k=1$, and 
\[
\lim_{n\to\infty}n^{-1}\!\!\!\sum_{0<k\le n-1}\!\!\!\!\rho^kE_{n,k}=\sum_{k\ge 0} \rho^k c_k <\infty.
\]
\end{proof}
From Lemma \ref{limH(t)n/nexists} it follows that $R_n$, the rank $R(v)$ of the uniformly random 
vertex $v$, converges in distribution to $R$, ($\pr(R=k)=c_k$, $k\ge 0$) fast enough for $\ex[\rho^{R_n}]$
to converge to $\ex[\rho^R]$ if $\rho<3/2$. The proof of Theorem \ref{Rn->R} is
complete.
\end{proof}
Next we will show that the ranks of a finite ordered tuple of the random vertices are mutually independent in the limit $n\to\infty$. 
\begin{theorem}\label{mutind} Let $t>1$ be fixed. For an ordered, fixed, $t$-tuple $\bold k=(k_1,\dots,k_t)$, let $p_n(\bold k)$ denote the probability that the uniformly random $t$-tuple of
vertices $(v_1,\dots,v_t)$ have ranks $R(v_1)=k_1,\dots,R(v_t)= k_t$. Then $\lim_{n\to\infty}p_n(\bold k)=\prod_{j=1}^t c_{k_j}$.
\end{theorem}
\begin{proof} For brevity, we consider $t=2$ only. Let $E_{n,\bold k}$ denote the expected 
number of ordered pairs of vertices with ranks $k_1$ and $k_2$ respectively; so $E_{n,\bold k}=
n(n-1)p_n(\bold k)$.  Then
\[
E_{n,\bold k}=E_{n,\bold k}^\prime+E_{n,\bold k}^{\prime\prime};
\]
here $E_{n,\bold k}^\prime$ is the contribution of the ordered pairs $(v_1,v_2)$ such that
$v_1$ is not a descendant of $v_2$, and $v_2$ is not a descendant of $v_1$. $E_{n,\bold k}^{\prime\prime}$ comes from the remaining pairs $(v_1,v_2)$. Obviously $E_{n,\bold k}^{\prime\prime}\le 2{\cal E}_n$, ${\cal E}_n$ being  the expected number of pairs $(v_1,v_2)$ such that $v_2$ is a descendant of $v_1$. Then, for $n>1$,
\[
{\cal E}_n = (n-1)+ \frac{2}{n}\sum_{j=0}^{n-1}{\cal E}_j\Longrightarrow {\cal E}_n =O(n\log n).
\]
Therefore $E_{n,\bold k}^{\prime\prime}=O(n\log n)$. Turn to $E_{n,\bold k}^\prime$. This time
\[
E_{n,\bold k}^\prime=\frac{2}{n}\sum_{j=0}^{n-1}E_{j,k_1}E_{n-j-1,k_2}+\frac{2}{n}\sum_{j=0}^{n-1}E_{j,\bold k}
^\prime;
\]
the first sum accounts for pairs $(v_1,v_2)$ such that $v_1$ and $v_2$ do not belong to the 
same subtree, which explains the product $E_{j,k_1}E_{n-j-1,k_2}$ of the expected (conditional)
counts of vertices of rank $k_1$ and of rank $k_2$, in the left subtree and the right subtree
respectively. We know that, for a fixed $k$, $E_{\nu,k}=\nu c_k +o(\nu)$ if $\nu\to\infty$. It
follows then easily that
\[  
\frac{2}{n}\sum_{j=0}^{n-1}E_{j,k_1}E_{n-j-1,k_2}=c_{k_1}c_{k_2}\frac{n^2}{3} +o(n^2).
\]
Therefore, for every $\eps>0$ there exists $A=A(\eps)>0$ such that
\begin{equation}\label{bn^+}
\frac{2}{n}\sum_{j=0}^{n-1}E_{j,k_1}E_{n-j-1,k_2}\le b_n^+:=\frac{n^2}{3}c_{k_1}c_{k_2}+\eps n^2+A.
\end{equation}
This implies $E_{n,\bold k}^\prime \le {\cal E}_{n,\bold k}^+$, where
\[
{\cal E}_{n,\bold k}^+=b_n^+ + \frac{2}{n}\sum_{j=0}^{n-1}{\cal E}_{j,\bold k}^+,\quad {\cal E}_{j,\bold k}^+=0,\,\,(j=0,1).
\]
So, as usual,
\[
{\cal E}_{n,\bold k}^+=(n+1)\sum_{j=2}^n\frac{jb_j^+ - (j-1)b_{j-!}^+}{j(j+1)};
\]
here, using \eqref{bn^+},
\[
\frac{jb_j^+ - (j-1)b_{j-!}^+}{j(j+1)}=\frac{(c_{k_1}c_{k_2}+3\eps)j^2 +O(j)}{j^2+O(j)}=
c_{k_1}c_{k_2}+3\eps+O({j-1}).
\]
Consequently
\[
{\cal E}_{n,\bold k}^+=\bigl[c_{k_1}c_{k_2}+3\eps\bigr]n^2 +O(n\log n).
\]
This implies
\[
\limsup \frac{E_{n,\bold k}^\prime}{n(n-1)}\le \lim\frac{{\cal E}_{n,\bold k}^+}{n^2}=c_{k_1}c_{k_2}+3\eps.
\]
Analogously,
\[
\liminf \frac{E_{n,\bold k}^\prime}{n(n-1)}\ge c_{k_1}c_{k_2}-3\eps.
\]
Letting $\eps\downarrow 0$, we obtain 
$
\lim \frac{E_{n,\bold k}^\prime}{n(n-1)}=c_{k_1}c_{k_2}.
$
Since $E_{n,\bold k}^{\prime\prime}=O(n\log n)$, we conclude that 
$\lim \frac{E_{n,\bold k}}{n(n-1)}=c_{k_1}c_{k_2}$.
\end{proof}

\begin{corollary}\label{concentr} Introduce $V_{n,k}$, the total number of vertices of
rank $k$; so $V_{n,0}=L_n$, the total number of leaves. Then $V_{n,k}/n\to c_k$  in
probability, i. e. for every $\eps>0$, $\pr(|V_{n,k}/n - c_k|>\eps)=o(1)$ as $n\to\infty$.
\end{corollary}
\begin{proof} We know that $\ex[V_{n,k}]/n=E_{n,k}/n\to c_k$, and we also know that $\ex[V_{n,k}(V_{n,k}-1)
/n(n-1)]\to c_k^2$. It remains to apply Chebyshev's inequality.
\end{proof}

That $o(1)$ in the Corollary would not be enough for us. Recall though that for $L_n:=V_{n,0}$ we were able to show (Corollary \ref{PXn<an}) that $V_{n,0}< (c_0-\eps)n$ with probability $\exp(-\eps n^{1-\delta})$ at most, smaller than $n^{-K}$ for all $K>0$.
We conjecture that the analogous property holds for all $V_{n,k}$.  A weaker
claim, analogously proved, will suffice for our needs in Section $3$.
\begin{lemma}\label{Lnk>epsn} For $\delta<1$ and $n$ large enough,
\[
\pr(V_{n,k}< 0.03an) \le \exp\bigl(-0.01an^{1-\delta}\bigr),\quad a:=1/k!.
\]
\end{lemma}
\begin{proof} {\bf (i)\/} Clearly $V_{n,k}\ge \cal V_{n,k}$, which is the total number of vertex-to-leaf
paths of length $k$ such that every non-leaf vertex of the path has only one child. Introduce 
$F(x,y)=\sum_{n\ge 0}y^n \ex\bigl[x^{\cal V_{n,k}}\bigr]$, $(\cal V_{0,k}:=0)$. For $y<1$, $F(1,y)=(1-y)^{-1}$;  so for $x\le 1$, $y<1$, we have $F(x,y)\le (1-y)^{-1}<\infty$. 

Now $\cal V_{n,k}=0$ for $n\le k$, $\cal V_{k+1,k}=1\, (0 \text{ resp.})$ with probability $2^k/(k+1)!$ ($1-
2^k/(k+1)!$, resp.), and for $n>k+1$,
\[
\ex\bigl[x^{\cal V_{n,k}}\bigr]=\frac{1}{n}\sum_{j=0}^{n-1} \ex\bigl[x^{\cal V_{j,k}}\bigr]\cdot 
\ex\bigl[x^{\cal V_{n-1-j,k}}\bigr].
\]
It follows after simple algebra that 
\begin{equation}\label{diff,F2}
\frac{\partial}{\partial y}F(x,y) = F^2(x,y) - (1-x)y^k\frac{2^k}{k!},
\end{equation}
blending with \eqref{diff,F1} for $k=0$. Consequently, for $y\ge 1/2$,
\[
\frac{\partial}{\partial y}F(x,y) \le  F^2(x,y) - a(1-x),\quad (a=1/k!)\,.
\]
Introduce $G(x,y)$, ($y\ge 1/2$), the solution of
\[
\frac{\partial}{\partial y}G(x,y) = G^2(x,y) - a(1-x),\quad G(x,1/2)=F(x,1/2).
\]
Integrating the last equation and using
\[
G^2(x,1/2)-a(1-x)=F^2(x,1/2)-a(1-x)>0, 
\]
we obtain that $G(x,y)$ exists for $y\in [1/2,y_1(x))$,
\begin{equation}\label{y1(x)=}
y_1(x):=1/2+\bigl(2\sqrt{a(1-x)}\bigr)^{-1}\,
\log\frac{F(x,1/2)+\sqrt{a(1-x)}}{F(x,1/2)-\sqrt{a(1-x)}},
\end{equation}
and it is given by
\begin{equation}\label{Gxy=}
G(x,y)=\sqrt{a(1-x)}\,\frac{1+\exp(\sqrt{a(1-x)}(2y-1))\frac{F(x,1/2)-\sqrt{a(1-x)}}{F(x,1/2)+\sqrt{a(1-x)}}}
{1-\exp(\sqrt{a(1-x)}(2y-1))\frac{F(x,1/2)-\sqrt{a(1-x)}}{F(x,1/2)+\sqrt{a(1-x)}}}.
\end{equation}
(So $G(x,y)$ blows up as $y\uparrow y_1(x)$.) Consequently $F(x,y)$ exists for $y<y_1(x)$,
and $F(x,y)\le G(x,y)$ for $y\in [1/2,y_1(x))$.

{\bf (ii)\/} Armed with \eqref{y1(x)=}-\eqref{Gxy=} and $F(x,y)\le G(x,y)$, 
we choose $x=e^{-n^{-\delta}}$ and $y=e^{0.04a n^{-\delta}}$,
which is strictly below $y_1(x)$ for $n$ large, as $F(x,1/2)\le 2$, and apply the Chernoff-type bound:
\begin{align*}
\pr(\cal V_{n,k}< 0.03an) &\le x^{-0.03an}y^{-n} F(x,y) \le x^{-an}y^{-n} G(x,y)\\
&=O\bigl(x^{-an}y^{-n}\bigr)\le e^{-0.01an^{1-\delta}}.
\end{align*}
\end{proof}
\section{Closer look at the distribution $\{c_k\}$.} In Theorem \ref{Rn->R} we proved existence of 
finite  $\lim_{n\to\infty}
n^{-1}\sum_k \rho^k E_{n,k}$ for $\rho<3/2$, which implied that $1-\sum_{j=0}^{k-1} c_k=O(q^k)$ for every 
$q>2/3$. Focusing exclusively on the sequence $\{c_k\}$, we prove a considerably
stronger bound.
\begin{theorem}\label{1/3^n} The inequality
 \[
 1-\sum_{j=0}^{k}c_j \le \frac{6k+7}{3}\left(\frac{1}{3}\right)^k
 \] holds.
 \end{theorem} 
 \begin{proof} 
{\bf (i)\/} For $n\ge 1$, $k\ge 0$, let $a_{n,k}$ be the 
total number of vertices of rank $k$ in all $n!$ permutations of $[n]$, and let $b_{n,k}$ be
the total number of permutations for which the root of the tree is of rank $k$. So 
$a_{n,k}/n!=E_{n,k}$, the expected number of rank $k$ vertices in the random tree, and $b_{n,k}/n!$ is the probability that its root is of rank $k$.
Introduce $A_k(x)=\sum_{n>0} x^na_{n,k}/n! $ and $B_k(x)=\sum_{n>0}x^n b_{n,k}/n!$;
in particular, $B_0(x)=x$. From Lemma 3.1, Lemma 3.2 (B\'ona \cite{protected}),
\begin{equation}\label{A'B'kequation}
\begin{aligned}
A_k^\prime(x)&=\frac{2}{1-x}\cdot A_k(x)+B_k^\prime(x),\quad (k\ge 0),\\
B_k^\prime(x)&=2B_{k-1}(x)\cdot\left(\frac{1}{1-x}-\sum_{j=0}^{k-2}B_j(x)\right)-B_{k-1}(x)^2,
\quad (k>0).
\end{aligned}
\end{equation}
Introduce $A_{\le k}(x)=\sum_{0\le j\le k}A_j(x)$ and $B_{\le k}(x)=\sum_{0\le j\le k}B_j(x)$;
in particular $A_{\le k}(x)$ is the generating function of $\{\sum_{j\le k} E_{n,j}\}_{n\ge 0}$.
 It follows from the equation \eqref{A'B'kequation} that 
\begin{align}
A_{\le k}^\prime(x)&=\frac{2}{1-x}\cdot A_{\le k}(x)+B_{\le k}^\prime(x),\quad (k\ge 0),\label{calAk'}\\
\frac{d}{dx}\left(\frac{1}{1-x}-B_{\le k}(x)\right)&=\left(\frac{1}{1-x}-B_{\le k-1}(x)\right)^2 -1,
\quad (k>0).
\label{calBk'}
\end{align}
The equation \eqref{calBk'} can also be obtained directly via the conditional independence argument. Here is how.  Let $p_{n,\le k}$ be the probability that the root rank is
$k$ at most, so $p_{n,>k}:=1-p_{n,\le k}$ is the probability that the root rank strictly exceeds $k$. Clearly $p_{n,\le k}=\sum_{j\le k} b_{n,j}/n!$, and therefore
$B_{\le k}(x)$ is the generating function of $\{p_{n,\le k}\}_{n\ge 1}$.
Then, for $n>1$ and $k\ge 0$,
\[
p_{n,>k} = (1/n) \sum_{j=0}^{n-1}p_{j,>k-1}\cdot p_{n-j-1,>k-1},
\]
where $p_{0,>k-1}:=1$, since conditioned on the left subtree having size $k$, the left subtree and the right subtree are independent. Consequently, as $p_{1,>k}=0$ for all $k \geq 0$,
\begin{align*}
\frac{d}{dx} \sum_{n\ge 1} p_{n,>k}x^n &=\left(\sum_{n\ge 0} p_{n,>k-1}x^n\right)^2 - 
p_{0,>k-1}p_{0,>k-1}\\
&=\left(\sum_{n\ge 0} p_{n,>k-1}x^n\right)^2 -1.
\end{align*}
Here
\begin{align*}
\sum_{n\ge 1} p_{n,>k}x^n&= \sum_{n\ge 1} (1 - p_{n,\le k}) x^n
\\
&= \frac{x}{1-x} - B_{\le k}(x)= 1/(1-x) -1 -B_{\le k}(x),
\end{align*}
and
\begin{align*}
\sum_{n\ge 0}  p_{n,>k-1}x^n &= 1 + \sum_{n\ge 1}p_{n,>k-1} x^n\\
&=1+\frac{x}{1-x}-B_{\le k-1}(x)\\
&= \frac{1}{1-x} - B_{\le k-1}(x),\qquad (B_{\le -1}(x):=0).
\end{align*}
So
\[
\frac{d}{dx}\left(\frac{1}{1-x} - B_{\le k}(x)\right) =\left(\frac{1}{1-x}- B_{\le k-1}(x)\right)^2 - 1.
\]
{\bf (ii)\/} From \eqref{calAk'}, it follows that, for $k>0$,
\begin{multline}\label{calAk,calBk}
A_{\le k}(x)=\frac{1}{(1-x)^2}\int_0^x (1-y)^2B_{\le k}^\prime(y)\,dy\\
=\frac{1}{(1-x)^2}\left[(1-x)^2 B_{\le k}(x) +2\int_0^x (1-y)B_{\le k}(y)\,dy\right],
\end{multline}
so for $x\uparrow 1$
\[
A_{\le k}(x)\sim \frac{2}{(1-x)^2} \int_0^1(1-y)B_{\le k}(y)\,dy.
\]
Since $A_{\le k}(x)$ is the generating function of $\{\sum_{j\le k}E_{n,j}\}_{n\ge 0}$ and
$n^{-1}\sum_{j\le k} E_{n,j}$ $\to \sum_{j=0}^k c_j$, it follows by the Tauberian theorem  that
\begin{equation}\label{ck=int}
\sum_{j=0}^{k} c_j=2\int_0^1(1-y)B_{\le k}(y)\,dy.
\end{equation}
Obviously
\[
B_{\le k}(x)=\frac{x}{1-x} -B_{>k}(x),
\]
where $B_{>k}(x)$ is the generating function of $\{b_{n,>k}/n!\}$, $b_{n,>k}$ being the number
of permutations such that the root  rank (strictly)  exceeds $k$. Consequently
\begin{equation}\label{1-ck=int}
1-\sum_{j=0}^{k} c_j=2\int_0^1(1-y)B_{>k}(y)\,dy.
\end{equation}
Thus, to bound $1-\sum_{j=0}^{k }c_j$ from above we need to bound $B_{>k}(x)$ from above.
Clearly $b_{n,>k}$ is bounded above by the number of permutations for which there exists
a root-to-leaf path of (edge) length exceeding $k$. A success of this approach depends on
how efficient would be our search for a path that has a good chance to be comparable in length to
the shortest path.\\

{\bf (ii)\/} Here is a randomized greedy algorithm with a plausibly good chance to find such a 
competitive path.
If there are two non-empty subtrees at the root of the tree, we {\it delete\/} a subtree with probability
proportional to the number of vertices in it. We repeat the same procedure at the root of
the remaining subtree, and continue until the remaining subtree is a leaf of the whole tree. The
resulting sequence of roots of the nested subtrees forms a root-to-leaf path in the whole tree. 

For $n\ge 1$, $k\ge -1$, let $\pi_{n,>k}$ denote the probability that the length of this path exceeds $k$; obviously $p_{n,>k}\le \pi_{n,>k}$. Further, $\pi_{n,>-1}=1$, and for $n>1$, $k\ge 0$,
\[
\pi_{n,>k}=\frac{2}{n}\,\pi_{n-1,>k-1}
+\frac{1}{n}\sum_{j=1}^{n-2}\left[\frac{n-1-j}{n-1}\,\pi_{j,>k-1}+\frac{j}{n-1}\,\pi_{n-1-j,>k-1}\right],
\]
or
\begin{equation}\label{g,recur}
(n)_2\pi_{n,>k}=2(n-1)\pi_{n,>k-1}+2\sum_{j=1}^{n-2}(n-1-j)\pi_{j,>k-1}.
\end{equation}
Introduce $P_{>k}(x)=\sum_{n>0}\pi_{n,>k} x^n$; in particular 
\[
P_{>-1}(x)=\sum_{n>0} x^n=\frac{x}{1-x}.
\]
 Obviously $B_{>k}(x)\le P_{>k}(x)$,
and so the equation \eqref{1-ck=int} yields
\begin{equation}\label{1-ck<int}
1-\sum_{j=0}^{k}c_j\le 2\int_0^1(1-y)P_{>k}(y)\,dy,\quad (k\ge 0).
\end{equation}
Since $\pi_{n,>k}=0$ for $n\le k$, we have $P_{>k}^{(t)}(0) =0$ for $t\le k$. It follows from \eqref{g,recur} that 
\begin{align*}
\frac{d^2P_{>k}(x)}{dx^2}&=\sum_{n\ge 2}(n)_2\pi_{n,>k}\\
&=2\sum_{n\ge 2}(n-1)\pi_{n-1,>k-1}x^{n-2}+2\sum_{n\ge 2}x^{n-2}\sum_{j=1}^{n-2}(n-1-j)\pi_{j,>k-1}
\\
&=2\,\frac{d}{dx}\sum_{\nu\ge 1} \pi_{\nu,>k-1}x^{\nu}+2\sum_{j\ge 1}\pi_{j,>k-1}x^j\sum_{n\ge j+2}
(n-1-j)x^{n-2-j}\\
&=2\,\frac{dP_{>k-1}}{dx}+2\left(\sum_{j\ge 1}\pi_{j,>k-1}x^j\right)\left(\sum_{\nu\ge 1}\nu x^{\nu-1}\right)
\\
&=2\,\frac{dP_{>k-1}}{dx}+\frac{2}{(1-x)^2} \,P_{>k-1}(x).
\end{align*}
Thus
\begin{equation}\label{lindifGk}
\frac{d^2P_{>k}(x)}{dx^2}=2\,\frac{dP_{>k-1}}{dx}+\frac{2}{(1-x)^2} \,P_{>k-1}(x),\quad 
(P_{>k}^{(r)}(0) =0\text{  for } r\le k). 
\end{equation}
In light of \eqref{1-ck<int} it  seems necessary, as before, to integrate successively the differential equations
\eqref{lindifGk} for $P_{>k^\prime}(x)$, $k^\prime=1,2,\dots,k$, and then to evaluate the RHS 
of the bound \eqref{1-ck<int}. In fact, that is how we computed the bounds  \eqref{1-ck<int}
for $k$ up to $10$; linearity of \eqref{lindifGk} was critical for success of this computation. The data showed, rather compellingly, that the bound decays faster than
$(1/2)^k$. In absence of any tractable expression for $P_{>k}(x)$ when $k$ is large, the
issue was to find a way to bound the integral in \eqref{1-ck<int} without such an expression.\\

{\bf (iii)\/} Linearity of \eqref{lindifGk} to the rescue again!  Introduce
\[
I_{k,t}=\int_0^1(1-y)^t P_{>k}(y)\,dy, \quad k\ge -1,\,t>0;
\]
so 
\begin{equation}\label{1-ck<2Ik1}
1-\sum_{j=0}^{k-1}c_k \le 2I_{k,1},\quad (k>0).
\end{equation}
Notice first that, for $t>0$,
\begin{align}
I_{-1,t}&=\int_0^1(1-y)^tP_{>-1}(y)\,dy=\int_0^1\bigl[-(1-y)^t+(1-y)^{t-1}\bigr]\,dt\notag\\
&=\frac{1}{t(t+1)}.\label{I0t=}
\end{align}
Let us show that, for $k\ge 0$, $t>0$,
\begin{equation}\label{recur,Ikt}
I_{k,t}=\frac{2}{(t+2)_2}\bigl[I_{k-1,t}+(t+2)I_{k-1,t+1}\bigr].
\end{equation}
Indeed, using $P_{>k}^{(r)}(0)=0$ for $r=0,1$ and \eqref{lindifGk},
\begin{align*}
I_{k,t}&=\int_0^1(1-y)^tP_{>k}(y)\,dy\\
&=\frac{1}{(t+2)_2}\int_0^1(1-y)^{t+2}\frac{d^2 P_{>k}(y)}{dy^2}\,dy\\
&=\frac{2}{(t+2)_2}\int_0^1(1-y)^{t+2}\left[\frac{dP_{>k-1}(y)}{dy}+\frac{P_{>k-1}(y)}{(1-y)^2}\right]\,dy\\
&=\frac{2}{(t+2)_2}\left[(t+2)\int_0^1(1-y)^{t+1}P_{>k-1}(y)\,dy+\int_0^1(1-y)^tP_{>k-1}(y)\,dy\right]\\
&=\frac{2}{(t+2)_2}\bigl[I_{k-1,t}+(t+2)I_{k-1,t+1}\bigr].
\end{align*}
In particular, 
\[
I_{k,1}=\frac{1}{3}\bigl[I_{k-1,1}+3I_{k-1,2}\bigr]\ge \frac{1}{3}I_{k-1,1},
\]
so that $I_{k,1}\ge \text{const }(1/3)^k$. We are about to show that in fact $I_{k,1}\le \text{const }(1/3)^k$,
i. e. $I_{k,1}$ is of order $(1/3)^k$ exactly.

To ths end,  fix $\tau>0$ and consider $I_{k,t}$ for $k\ge -1$ and $t\ge\tau$. Let us show that
\begin{equation}\label{Ikt,t=>tau}
I_{k,t}\le \frac{1}{(t+1)_2}\left(\frac{2}{\tau+2}\right)^{k+1}.
\end{equation}
By \eqref{I0t=}, the bound holds for $k=-1$. Inductively, if it holds for for some $k\ge 0$, then
by \eqref{recur,Ikt}
\begin{align*}
I_{k+1,t}&\le\frac{2}{(t+2)_2}\left[\frac{1}{(t+1)_2}\left(\frac{2}{\tau+2}\right)^{k+1}+\frac{t+2}{(t+2)_2}
\left(\frac{2}{\tau+2}\right)^{k+1}\right]\\
&=\frac{2}{(t+2)_3}\left(\frac{2}{\tau+2}\right)^{k+1}\\
&\le \frac{1}{(t+1)_2}\left(\frac{2}{\tau+2}\right)^{k+2}.
\end{align*}
So the bound \eqref{Ikt,t=>tau} is proven. In particular, 
\begin{equation*}
\text{for }t\ge 4,\quad I_{k,t}\le \frac{1}{(t+1)_2}\left(\frac{1}{3}\right)^{k+1}\Longrightarrow 
I_{k,4}\le 0.05 \left(\frac{1}{3}\right)^{k+1}.
\end{equation*}
Using the equation \eqref{recur,Ikt} for $t=3$, we have then
\begin{align*}
I_{k,3}=\frac{1}{10}I_{k-1,3}+\frac{1}{2}I_{k-1,4}
\le 0.1 I_{k-1,3}+0.025\left(\frac{1}{3}\right)^k.
\end{align*}
Iterating this recurrence inequality and using \eqref{I0t=} for $I_{0,3}$, we obtain
\begin{equation}\label{Ik3<}
\begin{aligned}
I_{k,3}&\le \frac{1}{3\cdot 4}\left(\frac{1}{10}\right)^{k+1}+0.025\left(\frac{1}{3}\right)^{k}\sum_{j\ge 0}
\left(\frac{3}{10}\right)^j\\
&=\left(\frac{1}{3}\right)^{k+1}\left(\frac{1}{12}+0.025\cdot\frac{30}{7}\right)\le\frac{1}{5}\left(\frac{1}{3}\right)^{k+1}.
\end{aligned}
\end{equation}
Analogously, using  the equation \eqref{recur,Ikt} for $t=2$ in conjunction with \eqref{Ik3<},
we iterate the resulting recurrence inequality
\[
I_{k,2}\le \frac{1}{6}I_{k-1,2}+\frac{2}{15}\left(\frac{1}{3}\right)^k .
\]
Recalling \eqref{I0t=} for $I_{0,2}$, we obtain
\begin{equation}\label{Ik2<}
I_{k,2}\le \left(\frac{1}{3}\right)^{k+1}.
\end{equation}
Lastly, combining \eqref{recur,Ikt} for $t=1$ and \eqref{Ik2<}, we have
\begin{equation*}
I_{k,1}\le \frac{1}{3}I_{k-1,1}+\left(\frac{1}{3}\right)^{k}.
\end{equation*}
 this recurrence, and using  \eqref{I0t=} for $I_{0,1}$, we arrive at
\begin{equation}\label{Ik1<}
I_{k,1}\le \frac{6k+7}{6}\left(\frac{1}{3}\right)^k.
\end{equation}
The bounds \eqref{Ik1<} and \eqref{1-ck<2Ik1} taken together imply
\[
1-\sum_{j=0}^{k} c_j \le \frac{6k+7}{3}\left(\frac{1}{3}\right)^k.
\]
\end{proof}

Next we prove a qualitatively matching lower bound for $1-\sum_{j=0}^{k}c_j$.
 Introduce the function $g(\alpha)= \alpha +\alpha \log(2/\alpha)-1$. The equation 
$g(\alpha)=0$ has two positive roots. Let $\alpha_0$ denote the smaller root; $\alpha_0\approx
0.373$. It was proved in  \cite{mahmoudpit} that the likely length of the shortest path from the root
of the random tree to a leaf is at lest $(\alpha_0-\varepsilon)\log n$, for every $\varepsilon>0$.
\begin{theorem}\label{exp,lower} There exists a positive constant $\gamma$ such that for all $k\ge 0$,
\begin{equation}\label{lower,alpha0}
 1-\sum_{j=0}^{k}c_j\ge \gamma e^{-k/\alpha_0}.
\end{equation}
\end{theorem}
\begin{proof} {\bf (a)\/} Given an integer $m$, consider the random tree on $[m]$. Let $S_m$
denote the edge length of the shortest path from the root to a leaf. Then, for every $s\in [0,m-1]$,
\[
\pr(S_m\le s)=\sum_{\mu\le s}\pr(S_m=\mu)\le\sum_{\mu\le s}\ex[X_{m,\mu}],
\]
where $X_{m,\mu}$ is the total number of leaves at distance $\mu$ from the root. By 
\eqref{(log n+1)^{k-1}},
\begin{equation*}
\ex[X_{m,\mu}]\le \frac{2^{\mu}}{(\mu-1)!}\frac{(\log m +1)^{\mu-1}}{m}.
\end{equation*}
Given $\gamma>0$, we have: for $\mu\le \gamma\log m$,
\begin{align*}
\ex[X_{m,\mu}]&\le \frac{\mu}{m(\log m+1)}\cdot \frac{2^{\mu}(\log m+1)^{\mu}}{\mu!}\\
&\le \frac{\gamma e^{\gamma}}{m}\left(\frac{2\log m}{\mu/e}\right)^{\mu}.
\end{align*}
As a function of $\mu$, the RHS increases for $\mu\le 2\log m$. Assuming that $\gamma\le 2$,
we obtain then that 
\begin{align*}
\pr(S_m\le \gamma \log m)&\le \frac{\gamma^2 e^{\gamma}\log m}{m}\cdot
\left(\frac{2e}{\gamma}\right)^{\gamma \log m}\\
&=\gamma^2e^{\gamma}\log m\cdot\exp[g(\gamma)\log m].
\end{align*}
Now $g(\alpha)$ is strictly increasing on $[0,\alpha_0]$, from $g(0)=-1$ to $g(\alpha_0)=0$.
So picking $\gamma=\alpha_0/2$ say, we obtain 
\begin{equation}\label{s << log m}
\pr(S_m\le (\alpha_0/2) \log m)\le (\alpha_0/2)^2e^{\alpha_0/2}\cdot m^{g(\alpha_0/2)}\log m=o(1).
\end{equation}
For $\alpha:=\mu/\log m\in (\alpha_0/2,\alpha_0)$ we have (see \cite{mahmoudpit})
\begin{equation}\label{E[Xmmu]}
\begin{aligned}
\ex[X_{m,\mu}]&=(1+\eps_m)K(\alpha) (\log m)^{-1/2}\exp[g(\alpha)\log m], 
\\
K(\alpha)&:=\bigl(\sqrt{2\pi\alpha}\,\Gamma(\alpha)\bigr)^{-1}\exp(\alpha -1),
\end{aligned}
\end{equation}
where $\lim_{m\to\infty} \eps_m=0$. By convexity of $g(\alpha)$ on $[0,\alpha_0]$,
\[
g(\alpha)\le g(\alpha_0)+(\alpha-\alpha_0)g^\prime(\alpha_0)=(\alpha-\alpha_0)g^\prime(\alpha_0),
\]
where $g^\prime:=g^\prime(\alpha_0)>0$. Therefore $\sum_{\alpha\in (\alpha_0/2,\alpha_0]}
\ex[X_{m,\mu}]$ is of order
\begin{equation*}
(\log m)^{-3/2}\!\!\int\limits_{\alpha_0g^\prime/2}^{\alpha_0g^\prime}\!\!\!e^{-x}\,dx=O\bigl((\log m)^{-3/2}\bigr).
\end{equation*}
Recalling \eqref{s << log m}, we conclude that
\begin{equation}\label{pr(Smsmall)}
\pr(S_m\le \alpha_0 \log m) =O\bigl((\log m)^{-3/2}\bigr).
\end{equation}
{\bf (b)\/} Given $\ell\ge 0$, let $Y_{n,\ell}$ denote the  total number of subtrees of size $m\ge\ell$. Then, for $n\ge\ell$,
\[
\ex[Y_{n,\ell}]=1+\frac{1}{n}\sum_{j=0}^{n-1}\bigl(\ex[Y_{j,\ell}]+\ex[Y_{n-1-j,\ell}]\bigr),
\]
with $\ex[Y_{j,\ell}]=0$ for $j<\ell$. The standard computation shows that
\begin{equation}\label{Ynell/n=}
\frac{\ex[Y_{n,\ell}]}{n+1}=\frac{2}{\ell+1}-\frac{1}{n+1}\Longrightarrow \frac{\ex[Y_{n,\ell}]}{n}=
(1+1/n)\left(\frac{2}{\ell+1}-\frac{1}{n+1}\right).
\end{equation}
Consider a generic subtree on $m\ge\ell$ vertices. Conditioned on its vertex set
 $\{p(i_1),\dots,p(i_m)\}$, $i_1<\cdots<i_m$, this subtree has the same distribution as the
 tree on $[m]$ grown from the uniformly random permutation of $[m]$. So, denoting 
 $S(p(i_1),\dots,p(i_m))$ the length of the shortest root-to-leaf path in this subtree, by 
 \eqref{pr(Smsmall)}, we have: uniformly for $m\ge\ell$,
\[
 \pr\bigl(S(p(i_1),\dots,p(i_m))>\alpha_0\log \ell\,\|\,p(i_1),\dots,p(i_m)\bigr)=1-O\bigl((\log \ell)^{-3/2}\bigr).
 \]
Let $Z_{n,\ell}$ denote the total number of the subtrees of size $m\ge\ell$ such that the
shortest root-to-leaf path has length exceeding $\alpha_0\log\ell$; clearly 
\[
\sum_{j\ge \alpha_0\log \ell}E_{n,k}\ge \ex[Z_{n,\ell}]. 
\]
From the above equation
it follows that
\[
\ex[Z_{n,\ell}\,\|\,Y_{n,\ell}]=\bigl[1-O\bigl((\log \ell)^{-3/2}\bigr)\bigr] Y_{n,\ell}.
\]
Combining this with  \eqref{Ynell/n=}, we obtain 
\[
\frac{\ex[Z_{n,\ell}]}{n}=\frac{2}{\ell+1}\bigl[1+O\bigl((\log \ell)^{-3/2}+n^{-1}\bigr)\bigr] .
\]
Therefore
\begin{align*}
\sum_{j\ge\alpha_0 \log \ell}c_j&=\lim_{n\to\infty}n^{-1}\sum_{j\ge\alpha_0\log \ell}E_{n,j}\\
&\ge \liminf_{n\to\infty}\frac{\ex[Z_{n,\ell}]}{n}\\
&=\frac{2}{\ell+1}\bigl[1+O\bigl((\log \ell)^{-3/2}\bigr)\bigr].
\end{align*}
Pick $k>0$ and set $\ell=\lceil e^{k/\alpha_0}\rceil$. Then the above estimate implies that
\[
\sum_{j>k}c_j\ge \frac{2}{\lceil e^{k/\alpha_0}\rceil+1}\bigl[1+O(k^{-3/2})\bigr]\ge
\frac{2}{3}e^{-k/\alpha_0}\bigl[1+O(k^{-3/2})\bigr].
\]
\end{proof}

{\bf Remark.\/} By Theorem \ref{1/3^n} and Theorem \ref{exp,lower} the radius of convergence of $\sum_k c_k x^k$ is in  the interval $[3,e^{1/0.373\dots}]$. What is the exact value of the radius? 
\section{Variations} 

Besides $E_{n,k}$, the expected counts of rank $k$ vertices, it is also natural to consider $F_{n,k}$ and $G_{n,k}$, the expected number of all pairs $(v, u)$, where $v$ is a vertex 
of rank $k$ and $u$ is a descendant leaf of $v$, and
 the expected number of all pairs $(v,u)$, where $v$ is a vertex 
of rank $k$ and $u$ is a {\it closest\/} descendant leaf of $v$.
 
Let us show  that, for each $k$, there exist finite 
\[
f_k=\lim_{n\to\infty}F_{n,k}/n,\quad g_k=\lim_{n\to\infty}G_{n,k}/n.
\]
 Consider $F_{n,k}$, for example. Introducing $f_{n,k}$, the expected product of the number of leaves of the random tree and the indicator of the event $\{\text{root rank}=k\}$, we have
\[
F_{n,k} = f_{n,k} +(1/n)\sum_{j=0}^{n-1} (F_{j,k}+F_{n-1-j,k}), \quad n>1.
\]
For  $n>0$, $f_{n,0}=0$;  for $k>0$,  using \eqref{Lj(x)},
\begin{equation}\label{(log n+1)^{k-1}}
\begin{aligned}
f_{n,k}&\le(n-1)\pr(\text{root rank}=k)\le n\ex[X_{n,k}]\\
&\le n 2^k [x^n] \frac{1}{k!}\left(\log\frac{1}{1-x}\right)^k
=2^k \frac{n}{n!}\, [y^{k-1}](y+1)\cdots(y+n-1)\\
&=2^k\frac{n (n-1)!}{n!}\left(\sum_{0<i_1<\cdots< i_{k-1}<n}\frac{1}{i_1\cdots i_{k-1}}\right)\\
&\le \frac{2^k}{(k-1)!}\left(\sum_{1\le i\le n-1}\frac{1}{i}\right)^{k-1}
\le\frac{2^k}{(k-1)!}\,(\log n+1)^{k-1}\\
&=O\bigl((\log n)^{k-1}\bigr).
\end{aligned}
\end{equation}
 So $x_n:=F_{n,k}$, $y_n:=f_{n,k}$ meet the
conditions of Lemma \ref{limexists} with $\eps\in (0,1)$. Consequently, for each $k$, there exists a finite $f_k:=\lim_{n\to\infty}F_{n,k}/n$.\\

To compute $f_k$, $g_k$, we need the recurrences similar to \eqref{calAk'}-\eqref{calBk'}. 
Introduce $f_{n,>k}=\sum_{j>k}f_{n,j}$,
and $\cal A_{k}(x)=\sum_{n\ge 1} x^n F_{n,k}$, $\cal B_{k}(x)=\sum_{n\ge 1}x^n f_{n,k}$,
and $\cal B_{>k}(x)=\sum_{n\ge 1}x^n f_{n,>k}$, where $f_{n,>k}:=\sum_{j>k}f_{n,j}$; so 
\begin{equation}\label{Bk=B(>k-1)-B(>k)}
\cal B_k(x)=\cal B_{>k-1}(x)-\cal B_{>k}(x).
\end{equation}
\begin{lemma}\label{F} For all nonnegative integers $k$, the following equalities hold. 
\begin{align}
\frac{d}{dx}\cal A_{k}(x)&=\frac{2}{1-x}\, \cal A_{k}(x)+\frac{d}{dx}\cal B_{k}(x),\label{F,eq1}\\
\frac{d}{dx}\cal B_{>k}(x)&=2\left(\frac{1}{1-x}-B_{\le k-1}(x)\right)\cal B_{>k-1}(x),
\label{F,eq2}
\end{align}
here $\{B_{\le t}(x)\}$ is the sequence determined by the recurrence \eqref{calBk'}, $B_{\le -1}(x):=0$, and $\cal B_{> -1}(x)=\cal B_{\ge 0}(x)$ is the generating function of the expected numbers of leaves, i. e.
\[
\cal B_{>-1}(x)=\frac{x-1}{3}+\frac{1}{3(1-x)^2}.
\]
Consequently
\begin{equation}\label{fk=}
f_k=2\int_0^1(1-x)\cal B_k(x)\,dx.
\end{equation}
\end{lemma}

Next, introduce $\widehat {A}_k(x)=\sum_{n\ge 1}x^n G_{n,k}$, $\widehat {B}_k(x)=\sum_{n\ge 1}
x^n g_{n,k}$, where $g_{n,k}:=\ex\left[1_{\{R(root)=k\}}\cal L_n\right]$ and $\cal L_n$ is the
number of leaves closest to the root of the tree.
\begin{lemma}\label{G}
The following equalities hold. 
\begin{align}
\frac{d}{dx}\widehat {A}_{k}(x)&=\frac{2}{1-x}\, \widehat {A}_{k}(x)+\frac{d}{dx}\widehat{B}_{k}(x),\quad (k\ge 0),\label{G,eq1}\\
\frac{d}{dx}\widehat {B}_k (x)&=2 \bigl[1+B_{\ge k-1}(x) \bigr]\widehat {B}_{k-1}(x), \quad (k> 0,\,\,\, \widehat{B}_0(x)=x). \label{G,eq2}
\end{align}
Consequently 
\begin{equation}\label{gk=}
g_k=2\int_0^1(1-x)\widehat{B}_k(x)\,dx.
\end{equation}
\end{lemma} 
\begin{proof} {\bf (I)\/}   Let $\text{{\it root\/}}$ denote the root of the random tree $T_n$ on $[n]$. Let $L_n$ denote the total number of
leaves of  $T_n$. For $n\ge 2$,  $L_n = L^{\prime} +L^{\prime\prime}$ where $L^\prime$ and $L^{\prime\prime}$ denote the total number of leaves in the left subtree $T^\prime$ and  the right subtree $T^{\prime\prime}$ respectively. 
Let $\text{{\it root\/}}^\prime$ ($\text{{\it root\/}}^{\prime\prime}$ resp.) denote the root of $T^\prime$ ($T^{\prime\prime}$ resp.) if this subtree is non-empty. If both subtrees are non-empty, then 
\[
1_{\{R(\text{{\it root\/}}>k)\}}=1_{\{R(\text{{\it root\/}}^\prime) >k-1\}}\cdot 1_{\{R(\text{{\it root\/}}^{\prime\prime}) >k-1\}},\quad (k\geq 0).
\]
Let $0<j<n-1$. Now, conditioned on the event ``the vertex set of $T^\prime$ is a given
set $J$ of $j$ elements from $[n]\setminus \text{{\it root\/}}$'', the subtrees $T^\prime$ and $T^{\prime\prime}$ are independent, and  marginally distributed as $T_j$ and $T_{n-1-j}$, respectively. So
\begin{align*}
&\ex\bigl[1_{\{R(\text{{\it root\/}})>k\}}L_n\| J\bigr]=\ex\bigl[1_{\{R(\text{{\it root\/}}^\prime) >k-1\}}1_{\{R(\text{{\it root\/}}^{\prime\prime}) >k-1\}}(L^\prime+L^{\prime\prime})\| J\bigr]\notag\\
=&\ex\bigl[1_{\{R((\text{root of }T_j) >k-1\}}1_{\{R((\text{root of }T_{n-1-j}) >k-1\}}(L_j+L_{n-1-j})\bigr]\notag\\
=&\ex\bigl[1_{\{R((\text{root of }T_j) >k-1\}}L_j\bigr] \pr(R((\text{root of }T_{n-1-j}) >k-1))\notag\\
&+\ex\bigl[1_{\{R((\text{root of }T_{n-1-j}) >k-1\}}L_{n-1-j}\bigr] \pr(R((\text{root of }T_j) >k-1))\notag\\
=&f_{j,>k-1}\cdot p_{n-1-j,>k-1} +f_{n-1-j,>k-1} \cdot p_{j,>k-1}.
\end{align*}
where $p_{\nu,>k-1}:=\pr(R(\text{root of }T_{\nu}) >k-1)$. Setting $f_{0,>k-1}=0$, $p_{0,>k-1}=1$,
we see that the last equality holds for $j=0,\,n-1$ as well. Since $|J|$ is uniform on $\{0,\dots,n-1\}$, we obtain then
\[
f_{n,>k}=\ex\bigl[1_{\{R(\text{{\it root\/}})>k\}}L_n\bigr]=\frac{2}{n}\sum_{j=0}^{n-1}f_{j,>k-1}\cdot p_{n-1-j,>k-1}.
\]
It follows immediately that 
\[
\frac{d}{dx}\sum_{n\ge 1}f_{n,>k}x^n=2\left(\sum_{n\ge 0} p_{n,>k-1}x^n\right)\cdot
\left(\sum_{n\ge 1}f_{n,>k-1} x^n\right),
\]
which is equivalent to \eqref{F,eq2}, since 
\begin{align*}
\sum_{n\ge 0} p_{n,>k-1}x^n&=1+\sum_{n\ge 1}(1-p_{n,\le k-1})x^n\\
&=1+\frac{x}{1-x}-B_{\le k-1}(x)=\frac{1}{1-x}-B_{\le k-1}(x).
\end{align*} 
The equation \eqref{F,eq1} is implied by a simple recurrence
\[
F_{n,k}=f_{n,k} +\frac{2}{n}\sum_{j=0}^{n-1} F_{j,k},\quad (n\ge 2,\,k\ge 0).
\]
Finally, from the equation \eqref{F,eq1},
\begin{align*}
f_k&=\lim_{x\uparrow 1}(1-x)^2\cal A_k(x)\\
&=\int_0^1(1-x)^2\frac{d}{dx}\cal B_k(x)\,dx=2\int_0^1(1-x)\cal B_k(x)\,dx.
\end{align*}The proof of Lemma \ref{F} is complete.\\

{\bf (II)\/} Let us prove the equation \eqref{G,eq2}. Recall that $\cal L_n$ denotes the total number of
leaves {\it closest\/} to the root of $T_n$. For $n\ge 2$, let $\cal L^\prime$ and $\cal L^{\prime\prime}$ denote the total number of leaves in the left subtree $T^\prime$, empty or not,
 ($T^{\prime\prime}$ resp.)  closest to its root. Let $k>0$. If both
subtrees are non-empty, i. e. $0<j<n-1$, then
\begin{align*}
&1_{\{R(\text{{\it root\/}}=k)\}}\cal L_n=1_{\{R(\text{{\it root\/}}^\prime) =k-1\}}\cdot 1_{\{R(\text{{\it root\/}}^{\prime\prime}) =k-1\}}\cal L_n\\
&+1_{\{R(\text{{\it root\/}}^\prime) =k-1\}}\cdot 1_{\{R(\text{{\it root\/}}^{\prime\prime}) >k-1\}}\cal L_n
+1_{\{R(\text{{\it root\/}}^\prime)>k-1\}}\cdot 1_{\{R(\text{{\it root\/}}^{\prime\prime}) =k-1\}}\cal L_n\\
=&1_{\{R(\text{{\it root\/}}^\prime) =k-1\}}\cdot 1_{\{R(\text{{\it root\/}}^{\prime\prime}) =k-1\}}
(\cal L^\prime+\cal L^{\prime\prime})\\
&+1_{\{R(\text{{\it root\/}}^\prime) =k-1\}}\cdot 1_{\{R(\text{{\it root\/}}^{\prime\prime}) >k-1\}}\cal L^\prime
+1_{\{R(\text{{\it root\/}}^\prime)>k-1\}}\cdot 1_{\{R(\text{{\it root\/}}^{\prime\prime}) =k-1\}}\cal L^{\prime\prime}.
\end{align*}
The contribution of the first product on the last RHS to $\ex\bigl[1_{\{R(\text{{\it root\/}})=k\}}\cal L_n\| J\bigr]$ is 
\[
g_{j,k-1}\cdot p_{n-1-j,k-1}+g_{n-1-j,k-1}\cdot p_{j,k-1}.
\]
The total contribution of the second product and the third product is
\[
g_{j,k-1}\cdot p_{n-1-j,>k-1} +g_{n-1-j,k-1}\cdot p_{j,>k-1},
\]
so that
\[
\ex\bigl[1_{\{R(\text{{\it root\/}})=k\}}\cal L_n\| J\bigr]=g_{j,k-1}\cdot p_{n-1-j,\ge k-1}+
g_{n-1-j,k-1}\cdot p_{j,\ge k-1}.
\]
The last formula continues to hold for $j=0$ and $j=n-1$, if we set $p_{0,\ge \ell}=1$ for all
$\ell\ge 0$.
Consequently
\[
g_{n,k}=\ex\bigl[1_{\{R(\text{{\it root\/}})=k\}}\cal L_n\bigr]=\frac{2}{n}\sum_{j=0}^{n-1}g_{j,k-1}\cdot
p_{n-1-j,\ge k-1},
\]
and \eqref{G,eq2} follows immediately. And, as before, the equation \eqref{G,eq1} is the
direct consequence of 
\[
G_{n,k}=g_{n,k}+\frac{2}{n}\sum_{j=0}^{n-1}G_{j,k}, \quad (n\ge 2,\,k\ge 0).
\]
The equation \eqref{gk=} is proved like the equation \eqref{fk=}.
This completes the proof of Lemma \ref{G}.
\end{proof}
Introduce $\cal L_{n,k}$ and $\widehat{L}_{n,k}$, the total number of descendant leaves
of rank $k$ vertices and the total number of descendant leaves {\it closest\/} to rank $k$ vertices. Recalling the notation $V_{n,k}$ for the total number of rank $k$ vertices, we see that
$\cal L_{n,k}/V_{n,k}$ and $\widehat{L}_{n,k}/V_{n,k}$ are  the average numbers of 
descendant leaves and the closest descendant leaves per vertex of rank $k$.
\begin{theorem}\label{pervertex} For all nonnegative integers $k$, the following equalities hold. 
\[
\lim_{n\to\infty}\ex\left[\frac{\cal L_{n,k}}{V_{n,k}}\right] = \frac{f_k}{c_k},\quad 
\lim_{n\to\infty}\ex\left[\frac{\widehat L_{n,k}}{V_{n,k}}\right] = \frac{g_k}{c_k}.
\]
\end{theorem}
\begin{proof} Consider  $\cal L_{n,k}/V_{n,k}$ for instance. Observe first that $\cal L_{n,k}\le n$. Now, for $a=1/k!$ and $\eps>0$, write
\begin{align*}
\ex\left[\frac{\cal L_{n,k}}{V_{n,k}}\right]&=\ex\left[\frac{\cal L_{n,k}}{V_{n,k}}1_{\{V_{n,k}<0.03an\}}\right]
+\ex\left[\frac{\cal L_{n,k}}{V_{n,k}}1_{\{V_{n,k}\ge 0.03an\}}1_{\{|V_{n,k}/n-c_k|>\eps\}}\right]\\
&\quad+\ex\left[\frac{\cal L_{n,k}}{V_{n,k}}1_{\{V_{n,k}\ge 0.03an\}}1_{\{|V_{n,k}/n-c_k|\le\eps\}}\right]\\
&=E_1+E_2+E_3.
\end{align*}
Here, by Lemma \ref{Lnk>epsn} and Corollary \ref{concentr} respectively,
\[
E_1 \le n e^{-0.01an^{1-\delta}} \to 0,\quad E_2=O\bigl(\pr(|V_{n,k}/n-c_k|>\eps)\bigr)\to 0,
\]
as $n\to\infty$, and
\begin{align*}
E_3&=\frac{1}{n(c_k +O(\eps))}\ex\left[\cal L_{n,k}1_{\{V_{n,k}\ge 0.03an\}}1_{\{|V_{n,k}/n-c_k|\le\eps\}}\right]\\
&=\frac{\ex[\cal L_{n,k}/n]}{c_k+O(\eps)}\left[1 +O\bigl(\pr(V_{n,k}<0.03an)+\pr(|V_{n,k}/n-c_k|>\eps)\bigr)
\right].
\end{align*}
Therefore
\[
\lim_{\eps\downarrow 0}\limsup_{n\to\infty} E_3=\lim_{\eps\downarrow 0}\liminf_{n\to\infty} 
E_3=f_k/c_k.
\]
So $\lim_{n\to\infty}\ex[\cal L_{n,k}/V_{n,k}] = f_k/c_k$.
\end{proof}

\noindent {\bf Note.\/} A slight modification of the proof of Corollary \ref{concentr} shows that, in probability,
$\cal L_{n,k}/n\to f_k$ and $\widehat{L}_{n,k}/n\to g_k$. Therefore
$\ex\bigl[\cal L_{n,k}/V_{n,k}\bigr]\to f_k/c_k$ and $\ex\bigl[\widehat L_{n,k}/V_{n,k}\bigr]\to g_k/c_k$ in probability as well.\\

Using Maple to integrate the differential equations \eqref{F,eq2} and \eqref{G,eq2}, we computed $\{f_j\}_{j\le 2}$ and $\{g_j\}_{j\le 2}$ via \eqref{fk=} and \eqref{gk=} respectively:
\begin{equation}\label{fk=,gk=}
\begin{aligned}
&f_0=\frac{1}{3}, \quad f_1 =\frac{17}{30}, \quad f_2=\frac{152389}{170100};\\
&g_0=\frac{1}{3}, \quad g_1=\frac{1}{3}, \quad\,\,\,\, g_2=\frac{49}{180}.
\end{aligned}
\end{equation}
Therefore
\begin{equation}\label{fk/ck=,gk/ck=}
\begin{aligned}
&\frac{f_0}{c_0}=1,\quad \frac{f_1}{c_1}=\frac{17}{9},\quad \frac{f_2}{c_2}=\frac{152389}{36141};\\
&\frac{g_0}{c_0}=1,\quad \frac{g_1}{c_1}=\frac{10}{9},\quad \frac{g_2}{c_2}=\frac{2205}{1721}.
\end{aligned}
\end{equation}
{\bf Remarks.\/} {\bf (i)\/} That $g_0$, $g_1$ should both be $1/3$ follows from the observation that, for
$n\geq 2$, the number of pairs $(v,u)$, where $v$ is a rank $k$ vertex and $u$ is its closest
descendant-leaf, is the same number of all leaves when $k=0$ or $k=1$. {\bf (ii)\/} The data
suggest that both $f_k/c_k$ and $g_k/c_k$ increase with $k$, albeit at a slower rate for $g_k/c_k$.

\section{Numerics and gap-free factorization conjecture} 

In conclusion we present some intriguing experimental data on number-theoretic properties of $\{c_k\}$. Recall that
\begin{equation}\label{ck=int}
\sum_{j=0}^{k} c_j=2\int_0^1(1-y)B_{\le k}(y)\,dy.
\end{equation}
Then, using using  \eqref{calBk'} ,
\begin{equation}\label{1integ}
\begin{aligned}
\int_0^1(1-y)B_{\le k}(y)\,dy&=\frac{1}{2}\int_0^1(1-y)^2 B_{\le k}(y)^\prime\,dy\notag\\
&=\frac{1}{2}\int_0^1\left[2-2y+y^2-\bigl(1-(1-y) B_{\le k-1}(y)\bigr)^2\right].
\end{aligned}
\end{equation}
so 
\begin{equation}\label{ck=int,better}
\sum_{j=0}^{k}c_j=\int_0^1\left[2-2y+y^2-\bigl[1-(1-y) B_{\leq k-1}(y)\bigr]^2\right]\,dy.
\end{equation}
The equation \eqref{ck=int,better} allows to compute $c_{k}$ directly through $B_{\le k-1}(x)$,
without knowing $B_k(x)$.

Using this simplification, we have obtained the exact values of $c_4$ and $c_5$. 
That is, we have computed that $c_4$ equals
\[\frac{122058464141653662196290113232646304412999902283512425580156787323}{3353377025022449199852900725670960067418280803797231788288000000000},\]
a fraction whose denominator has 67 digits, and whose approximate value is 0.0364. Combining this with the values
of $c_i$ for $i\leq 4$, we see that (with high probability) about 99.14 percent of all vertices are of rank four or less. 

The prime factorization of the denominator $denom (c_4)$, when $c_4$ is written in simplest terms, obtained by Maple, is even more interesting, since its factorized representation 
is 
\[denom(c_4)=  2^{17}\cdot 3^{18} \cdot 5^9 \cdot 7^8 \cdot 11^8 \cdot 13^7 \cdot 17^6 \cdot 19^5 \cdot 23^4 \cdot 29^2
\cdot 31.\]
So the largest prime divisor of  $denom(c_4)$ is 31, which is a tiny number compared to $denom(c_4)$. Even more striking is the fact
that $denom(c_4)$ is divisible by {\em every prime} up to 31. In stark contrast, 
the numerator of $c_4$, while comparable in size to the denominator,  is the product of just 
{\it two\/} primes, the smaller of which is  $232196467$. 

This surprising fact warrants a second look at the numbers $c_k$ for $k\leq 3$, already computed in B\'ona \cite{protected}. Here is the factorized representation of the denominators, 
including $denom(c_4)$:
\begin{itemize}
\item $denom(c_0)=3$,
\item $denom(c_1)=2\cdot 5$,
\item $denom(c_2)=2^2 \cdot 3^4 \cdot 5^2$, 
\item $denom(c_3)=2^8 \cdot 3^7 \cdot 5^5 \cdot 7^3 \cdot 11^3 \cdot 13^2 \cdot 17$, and
\item $denom(c_4)= 2^{17}\cdot 3^{18} \cdot 5^9 \cdot 7^8 \cdot 11^8 \cdot 13^7 \cdot 17^6 \cdot 19^5 \cdot 23^4 \cdot 29^2 \cdot 31$.
\end{itemize}
So  all $denom(c_k)$ for $k\leq 4$ have very small prime divisors. With the exception of $k=0$ and $k=1$, 
it seems that the prime divisors of $denom(c_k)$ are {\em precisely} the first $t$ prime numbers for some $t$. Those two exceptions may be a reflection of
how relatively simple the counting of leaves and their fathers is.

Even though the computation of $c_4$ was already exceptionally time-\linebreak consuming, we decided
to compute the next value $c_5$. This task turned out to be so problematic that time and again we were
tempted to give up. Mobilizing all the insight into the algebraic form of the functions $B_j(x)$, we eventually got the answer. The approximate value of 
$c_5$ is 0.0074.  So--with high probability--about 99.875 percent of all vertices are of rank five or less.  The number $denom(c_5)$ has $274$ digits, and its prime factorization is
\[2^{48}\cdot 3^{42} \cdot 5^{28} \cdot 7^{18} \cdot 11^{16} \cdot 13^{16} \cdot 17^{17} \cdot 19^{16} \cdot 23^{15} \cdot 29^{12} \cdot 31^{12} \cdot 37^{10} \cdot 41^9 \cdot 43^8 \cdot 47^7 \cdot 53^5 \cdot 59^3 \cdot 61^2.\]
If not for this strikingly simple factorization, we would not dare to type in the $274$-long monster.
So yet again, $denom(c_k)$ has only very small prime factors, and it is divisible by every prime up to its largest prime
factor, $61$. (As for the numerator, its smallest prime divisor must be extremely large as 
Maple-based factorization algorithm failed the task.)

 Based on these data points, we formulate the following conjectures. 

\begin{conjecture} \label{firstcon} Let $denom(c_k)$ be the denominator of $c_k$ when $c_k$ is written in smallest terms. 
Then the largest prime divisor of the denominator is at most as large as some relatively slowly growing function of $k$, possibly
$2^{k+1}+1$. 
\end{conjecture}

\begin{conjecture} \label{secondcon} Let $k\geq 2$, and let $p_k$ be largest prime divisor of  $denom(c_k)$. Then $denom(c_k)$ is divisible by every prime less than $p_k$. 
\end{conjecture}

Perhaps it is also true that the smallest prime divisor of the numerator of $c_k$ grows 
super-exponentially with $k$, but we hesitate to make any specific guess.  We are able to prove Conjecture \ref{firstcon} but not Conjecture \ref{secondcon}. The reason the second conjecture is out of reach for now is simple: the 
numerator of $c_k$ is a sum of a very large set of summands, and we are unable to prove that  sum will not be
divisible by at least as high a power of a given prime $p$ as the denominator of $c_k$ is. 

In order to prove Conjecture \ref{firstcon}, we will need a few simple technical lemmas.  Recall that 
$B_k(x)$ denotes the exponential generating function for the numbers of trees on vertex set $[n]$ whose root is 
of rank $k$. The first two  examples are $B_0(x)=x$, and $B_1(x)=2\log(1/(1-x)-2x-x^3/3$.

\begin{lemma} For all natural numbers $k$, we have $B_k(x) \in {\bf PL}$, meaning that
$B_k(x)$ is a bivariate polynomial $P_k(u,v)$, at $u=(1-x)$, $v=\log 1/(1-x)$.
\end{lemma}

\begin{proof} See Lemma 4.1 in \cite{protected}. \end{proof}

It is also proved in \cite{protected} that the class {\bf PL} is closed under integration. In fact, the following, 
stronger statement is true. 

\begin{lemma} \label{induction} 
Let $b$ and $c$ be non-negative integers, and let us write \[\int (1-x)^b \log \left (\frac{1}{1-x} \right) ^c \ dx
= \sum_{i=1}^{m}  a_i  (1-x)^{b_i} \log \left ( \frac{1}{1-x} \right)^{c_i} ,\]
with the rational numbers $a_i$ written in their simplest form. Then for all $i$, the denominator of $a_i$ has no 
prime divisor larger than $b+1$. 
\end{lemma} 

\begin{proof} 
This follows by induction on $c$, the  inital case of $c=0$ being obvious. 
Indeed, integration by parts yields
\begin{equation}\label{partint} 
\begin{aligned}
\int (1-x)^{b} \log \left ( \frac{1}{1-x} \right)^{c} dx & = - \log \left ( \frac{1}{1-x} \right)^{c}  \cdot 
\frac{(1-x)^{b+1}}{b+1}  \\ + & \int \frac{(1-x)^{b}}{b+1} \cdot c \log \left ( \frac{1}{1-x} \right)^{c-1} dx ,
\end{aligned} 
\end{equation}
and the proof is complete. 
\end{proof}
{\bf Note.\/} The equation \eqref{partint} implies
\begin{equation*}
I_{b,c}:=\int_0^1(1-x)^b\log \left ( \frac{1}{1-x} \right)^{c} dx =\frac{1_{\{c=0\}}}{b+1}+\frac{c}{b+1}I_{b,c-1},
\end{equation*}
so iterating the same operation, we obtain
\begin{equation}\label{Ibc}
I_{b,c}=\frac{c!}{(b+1)^{c+1}}.
\end{equation}
\begin{lemma} \label{upperbounds} 
When written in simplest form, no term of $B_k(x)$ has a denominator with a prime divisor larger than $2^{k+1} -1$.
Furthermore,  both the exponent $b_i$ of $(1-x)$ and the exponent $c_i$ of $\log(1/(1-x))$ in the ${\bf PL}$ form of $B_k(x)$ are at most
as large as 
$2^{k+1} -1$.
\end{lemma}

\begin{proof} We prove the Lemma by stong induction on $k$. It is straightforward to check that $B_0(x)$ and
$B_1(x)$ satisfy both requirements. Now let us assume that  the claims of the Lemma are true for all $B_j(x)$ with $j<k$, and prove 
prove them for $B_k$. Formula (\ref{A'B'kequation}) shows that  
 $B_k'(x)$ is a quadratic form of  $B_i(x)$ with $i<k$ and $(1-x)^{-1}$. Consequently $B_k'(x)$ is of the form
$\sum_{i=1}^{m}  a_i  (1-x)^{b_i} \log \left ( \frac{1}{1-x} \right)^{c_i}$, where $b_i\geq -1$ is an integer, 
while $a_i$ is rational and $c_i$ is a nonnegative integer. Moreover, it follows from \eqref{A'B'kequation} and the induction hypothesis that,
 in the sum representing $B_k'(x)$, both the exponent $b_i$ of $(1-x)$ and the exponent $c_i$ of
 $\log(1/(1-x))$ are at most as large as $2(2^k-1)=
2^{k+1}-2$.

Now the contribution of $\sum_{i:b_i=-1} a_i(1-x)^{b_i}\log(\tfrac{1}{1-x})^{c_i}$ to $B_k(x)$ 
itself is
\[
\sum_{i:b_i=-1}\frac{a_i}{c_i+1}\log\left(\frac{1}{1-x}\right)^{c_i+1},
\]
with $c_i+1\le 2^{k+1}-1$. As for the contribution to $B_k(x)$ of the remaining summands with
$b_i\ge 0$, using Lemma \ref{induction} and by (\ref{partint}), we see that in all the summands 
neither the exponent of $(1-x)$ nor the exponent of $\log\tfrac{1}{1-x}$ can exceed
$2^{k+1}-1$, since integration of the terms with $b_i\ge 0$, $c_i\ge 0$ will increase these exponents by at most one.
 As addition and multiplication of terms will not result in the appearance of a larger prime divisor, the claim for $B_k(x)$ is proved. 
\end{proof}

\begin{proof} (of Conjecture \ref{firstcon})
Recall that (\ref{A'B'kequation}) implies that
\[
c_k=\lim_{x\uparrow 1} (1-x)^2A_k(x)= 2\int_0^1(1-x) B_k(x) \, dx.
\]
Here 
\[
B_k(x)=\sum_ia_i (1-x)^{b_i}\left(\log\frac{1}{1-x}\right)^{c_i},
\]
$0\le b_i,\,c_i\le 2^{k+1}-1$, and no $a_i$ has denominator with a prime divisor larger
than $2^{k+1}-1$. From \eqref{Ibc} it follows then that $c_k$ is
the sum of rational fractions, whose denominators do not have prime divisors exceeding
$2^{k+1}+1$, which is a common upper bound for the largest denominator of $a_i$ and for the largest $b_i+2$. 
\end{proof}

\begin{center}  {\bf Acknowledgement}  \end{center}

The authors are thankful to Frank Garvan who advised them on numerous occasions on how to convince Maple to carry out a difficult task.

\end{document}